\documentclass{amsart}
\usepackage{amsmath,amsfonts,euscript,amscd,amsthm,amssymb,upref,graphics}
\usepackage{mathrsfs}

\usepackage[all]{xy}
\usepackage{tikz-cd}
\usepackage{color,soul}

\usepackage{latexsym}
\usepackage{amsfonts,amssymb} 
\usepackage{graphicx}
\usepackage{tikz}
\usepackage{caption}
\usepackage{subcaption}
\usepackage{enumitem}

\newcommand{\field}[1]{\mathbb{#1}}
\newcommand{\A}{\field{A}}
\newcommand{\C}{\field{C}}

\newcommand{\G}{\field{G}}

\newcommand{\N}{\field{N}}
\newcommand{\PP}{\field{P}}

\newcommand{\V}{\field{V}}
\newcommand{\Z}{\field{Z}}

\newcommand{\krn}{{\rm ker}\,}

\theoremstyle{plain}

\newtheorem{theorem}{Theorem}[section]

\newtheorem{lemma}[theorem]{Lemma}
\newtheorem{corollary}[theorem]{Corollary}
\newtheorem{definition}[theorem]{Definition}
\newtheorem{remark}[theorem]{Remark}

\newtheorem{example}[theorem]{Example}
\newtheorem{question}[theorem]{Question}

\theoremstyle{definition}

\theoremstyle{remark}

\begin{document}

\makeatletter	   
\makeatother     

\title{Presentations, embeddings and automorphisms of homogeneous spaces for $SL_2(\C )$}
\author{G. Freudenburg}
\date{\today}

\begin{abstract} For an algebraically closed field $k$ of characteristic zero and a linear algebraic $k$-group $G$, it is well known that every affine $G$-variety admits a $G$-equivariant closed embedding into a finite-dimensional $G$-module.  
Such an embedding is a {\it presentation} of the $G$-variety, and a {\it minimal presentation} is one for which the dimension the $G$-module is minimal. 
The problem of finding a minimal presentation
generalizes the problem of determining whether a group action on affine space is linearizable. 
We give a minimal presentation for each homogeneous space for $SL_2(k)$. 
This constitutes the paper's main work. 
Of particular interest are the surfaces $Y=SL_2(k)/T$ and $X=SL_2(k)/N$ where $T$ is the one-dimensional torus and $N$ is its normalizer. 
We show that the minimal presentation of $X$ has dimension 5, the embedding dimension of $X$ is 4, and there does not exist a closed $SL_2$-equivariant embedding of $X$ in $\A_k^4$. Thus, the $SL_2$-action on $X$ is {\it absolutely nonextendable} to $\A_k^4$. We give two other examples of surfaces with absolutely nonextendable group actions. 
In addition, $X$ is noncancelative, that is, there exists a surface $\tilde{X}$ such that $X\times\A_k^1\cong_k\tilde{X}\times\A_k^!$ and $X\not\cong_k\tilde{X}$. 
Finally, we settle the long-standing open question of whether there exist inequivalent closed embeddings of $Y$ in $\A_k^3$ by constructing inequivalent embeddings. 
\end{abstract}

\maketitle

\section{Introduction}
Let $k$ be an algebraically closed field of characteristic zero, and let
the triple $(V,G,\rho )$ consist of an irreducible affine $k$-variety $V$ equipped with an algebraic action $\rho :G\to {\rm Aut}_k(V)$ of a linear algebraic $k$-group $G$.
Then $V$ admits
a $G$-equivariant closed embedding into a finite-dimensional $G$-module; see \cite{Borel.91}, Proposition 1.12 and \cite{Kraft.84}, II.2.5.
Blanc, Furter and Poloni give a more precise result for smooth rational affine curves over $k=\C$, showing that every such curve $\Gamma$ 
admits a $G$-equivariant closed embedding into a $G$-module of dimension 3, where $G={\rm Aut}_k(\Gamma)$  \cite{Blanc.Furter.Poloni.16}.
Each curve $\Gamma$ also admits a closed planar embedding $f:\Gamma\to\A_k^2$. If $G$ is finite and $f$ is $G$-equivariant,
then $G\subset GL_2(k)$, since every action of a finite group on the plane is linearizable. 
The finite subgroups of $GL_2(k)$ are well-known; see \cite{Riemenschneider.77}. 
The authors identify curves $\Gamma$ where $G$ is finite, but $G$ is not a subgroup of $GL_2(k)$, meaning that there is no $G$-equivariant closed embedding of $\Gamma$ in $\A^2_k$. 
For example, there exist $\Gamma$ whose automorphism group is the alternating group $A_4$. 
Such considerations motivate the following definitions.
\begin{definition} 
A {\bf presentation} of $(V,G,\rho )$ is an equivariant surjection $\varphi : k[W]\to \mathcal{O}(V)$ for some $G$-module $W$. 
$\varphi$ is a {\bf minimal} presentation if and only if $W$ is of minimal dimension.
\end{definition} 

Here, $k[W]$ is the symmetric algebra of $W$, and $\mathcal{O}(V)$ is the coordinate ring of $V$.
\begin{definition}
$(V,G,\rho )$ is {\bf absolutely nonextendable} if there is no $G$-equivariant closed algebraic embedding of $V$ into $\A_k^n$, 
where $n$ is the embedding dimension of $V$. 
\end{definition}

For presentations, the condition that $G$ is an algebraic group is necessary. Derksen, Kutzschebauch and Winkelmann showed the following.
\begin{theorem} {\rm (\cite{Derksen.Kutzschebauch.Winkelmann.99}, Theorem 2)} 
For all $n\in\N$, there does not exist any effective differentiable,
or holomorphic, or algebraic action of the group 
\[
{\rm Aut}_{\C}(\C^*\times\C^*)\cong \left( O_2(\C)\times O_2(\C)\right)\rtimes SL_2(\Z)
\]
on $\A_{\C}^n$. 
\end{theorem}
The authors construct several examples of an algebraic group acting on a hypersurface in such a way that the action does not extend to the ambient space. 
Their methods only apply to hypersurfaces, and each example is based on a particular embedding of the variety. 

We begin with two examples of surfaces with absolutely nonextendable actions, each based on the following result. These surfaces have embedding dimension three. We construct a third example in {\it Section\,\ref{SL2-nonextend}}, using a homogeneous surface for $SL_2(\C)$ whose embedding dimension is four (i.e., not a hypersurface). 

\begin{theorem}\label{linearizable1} {\rm (\cite{Kraft.Russell.14}, Theorem A) }
   Every faithful action of a non-finite complex reductive group on $\A_{\C}^3$ is linearizable. 
\end{theorem}

\begin{example} {\rm Consider $(V,G,\rho )$ where: 
\[
V=\C^*\times\C^* \quad\text{and}\quad 
G={\rm Aut}_{\C}(\C^*)\times {\rm Aut}_{\C}(\C^*)
=O_2(\C)\times O_2(\C)
\]
Let $W$ be the irreducible $O_2$-module of dimension two, where $\C^*$ acts with weights $\pm 1$. Let 
\[
W\oplus W=(\C x\oplus\C y)\oplus (\C z\oplus\C w)
\]
and define the $G$-equivariant surjection:
\[
\varphi :\C [W\oplus W]\to \C[W\oplus W]/(xy-1,zw-1)\cong_{\C}\mathcal{O}(V)
\]
This is a minimal presentation of $(V,G,\rho )$, since there is no faithful representation of $G$ on $\C^3$. On the other hand, $V$ admits closed embeddings in $\C^3$, for example, the hypersurface $xyz=1$. Let $S\subset \C^3$ be a hypersurface isomorphic to $V$, and suppose that the $G$-action on $S$ extends to $\C^3$.
By {\it Theorem\,\ref{linearizable1}}, this $G$-action is linear in some system of coordinates, and also faithful (since it is faithful on $S$), which gives a contradiction. Therefore, the $G$-action on $V$ is absolutely nonextendable. 
}    
\end{example}

\begin{example} {\rm For each positive odd integer $n$, 
let $W_n$ be the $O_2$-module of dimension four defined by:
\[
^{\zeta}(a,b,x,y)=(\zeta^{-2}a,\zeta^2b,\zeta^{-n}x,\zeta^ny)\,\,\ ,\,\, \zeta\in\C^*\quad\text{and}\quad \tau (a,b,x,y)=(b,a,y,x)
\]
The ideal $I_n=(ab-1,x+a^ny)$ is invariant, and for each $n$ we have:
\[
\C[W_n]/I_n\cong \mathcal{O}(V) \quad\text{where}\quad V\cong \C^*\times\C
\]
We thus obtain a family $\{\rho_n\}$ of faithful $O_2$-actions on $V$ defined by
\[
^\lambda (t,s)= (\lambda^{-2}t,\lambda^ns) \,\, ,\lambda\in\C^* \quad\text{and}\quad  ^\mu (t,s)=(t^{-1},t^ns) 
\]
with presentation $\varphi_n:\C[W_n]\to \C[W_n]/I_n$. The function $x+a^ny\in I_n$ is a variable of $\C[W_n]$, but not an $O_2$-invariant. This yields a sequence of closed embeddings $\mathcal{V}(I_n)\subset\mathcal{V}(x+a^ny)\subset\C^4$ where $\mathcal{V}(I_n)\cong V$ is an invariant surface but the hyperplane $\mathcal{V}(x+a^ny)\cong\C^3$ is not invariant. 

Let $S\subset\C^3$ be a hypersurface isomorphic to $V$, and suppose that $\rho_n$ extends to a faithful $O_2$-action on $\C^3$. 
By {\it Theorem\,\ref{linearizable1}}, this action is linear in some system of coordinates. There are two types of faithful representations of $O_2(\C)$ on $\C^3$, given by 
$(\lambda , \alpha )$ and $(\lambda ,\beta )$ where $\lambda\in\C^*$ and $\alpha^2=\beta^2=1$ for:
\[
\lambda =\begin{pmatrix} \lambda^m & 0 &0\cr 0&\lambda^{-m}&0\cr 0&0&1\end{pmatrix} 
\,\, (m\ge 1)\,\, ,\,\, 
\alpha = \begin{pmatrix} 0&1&0\cr 1&0&0\cr 0&0&1\end{pmatrix}\,\, ,\,\, 
\beta = \begin{pmatrix} 0&1&0\cr 1&0&0\cr 0&0&-1\end{pmatrix} 
\]
The set of fixed points of $\alpha$ is the coordinate plane $x=y$, and the set of fixed points of $\beta$ is the coordinate line $x=y, z=0$. The set of fixed points of $\mu$ in $V$ consists of the union of a line and an isolated point. Therefore, 
$\rho_n$ can only extend to an action of type $(\lambda ,\alpha)$. However, since $S$ is a hypersurface, every irreducible component of the intersection of a coordinate plane with $S$ is a curve, giving a contradiction. It follows that, for each odd $n\ge 1$, $(V,O_2(\C),\rho_n)$ is absolutely nonextendable and $\varphi_n$ is a minimal presentation. 
}    
\end{example}

Recall that a {\bf homogeneous space} for the group $G$ is an algebraic variety equipped with a transitive algebraic $G$-action. One of our main goals is to find minimal presentations of homogeneous spaces for $SL_2(k)$.
In {\it Section\,\ref{prelim}} 
we give an algorithmic method to find a presentation for any affine $SL_2$-variety $V$ using the theory of fundamental pairs. 
By this method, we find a minimal presentation of each homogeneous space for $SL_2(k)$; see {\it Theorem\,\ref{strong}}. 
Each is a quotient $SL_2(k)/H$ for some closed reductive subgroup $H$, 
and if $B=\mathcal{O}(SL_2(k))$, then $\mathcal{O}(SL_2(\C)/H)=B^H$. 
In order to find presentations, we do not use standard tools of invariant theory such as the Reynolds operator to find $B^H$. 
Instead, we use a fundamental pair $(D,U)$ of locally nilpotent derivations of $B$
to extend each ring of binary invariants $\krn D=k[x,y]^H$ to $B^H$ via integration by $U$. We use Felix Klein's description of these invariant rings. 
For example, integrating elements of $k[x,y]^{\Z_3}=k[x^3,xy,y^3]$ yields the subring $B^{\Z_3}$ generated by 4+3+4=11 elements (corresponding to degrees 3,2,3), but we reduce this number to four. 
We thus realize $SL_2(k)/\Z_3$ as a fiber of the quotient morphism $\pi :\A_k^4\to\A_k^1$ for the irreducible $SL_2$-module of degree three. 
In particular, $SL_2(k)/\Z_3$ is equivariantly isomorphic to a hypersurface in $\A_k^4$. 
Likewise, the binary icosahedral group ${\rm BI}_{120}$ is a subgroup of order 120.
Integrating elements of $k[x,y]^{{\rm BI}_{120}}=k[f,g,h]$ yields the subring $B^{{\rm BI}_{120}}$ generated by 13+21+31=65 elements (corresponding to degrees 12, 20, 30 of $f,g,h$), but we reduce this number to 13. 

Let $T\subset SL_2(\C )$ be the one-dimensional torus and let $N$ be its normalizer. Then $Y:=SL_2(\C)/T$ and $X:=SL_2(\C)/N$ are the only 
homogeneous surfaces for $SL_2(\C)$. This was shown by Popov \cite{Popov.73a,Popov.73b}. 
Note that $Y\cong\PP^1\times\PP^1\setminus\Delta$ (diagonal $\Delta$) and $X\cong \PP^2\setminus C$ (conic $C$). 
The automorphism groups of $Y$ and $X$ were described by Danilov and Gizatullin  \cite{Danilov.Gizatullin.77}. 
We give this description in {\it Section\,\ref{cylinder}}. 
The standard embedding of $Y$ in $\C^3$ is the hypersurface $Q=\{ (x,y,z)\,\vert\, xy-z^2=1\}$ and  
the automorphism group of $Y$ is the subgroup of automorphisms of $\C^3$ restricting to $Q$.
Consequently, there can be no absolutely nonextendable $G$-action on $Y$.
In contrast, we prove:

\begin{theorem}\label{non-extend1} Let $\rho :SL_2(\C)\to {\rm Aut}_{\C}(X)$ be the nontrivial $SL_2$-action on $X=SL_2(\C)/N$. 
\begin{itemize}
\item [{\bf (a)}] Any minimal presentation of $\rho$ has the form $\varphi :k[\V_4]\to\mathcal{O}(X)$ where $\V_4$ is the irreducible $SL_2$-module of dimension 5. 
\item [{\bf (b)}] The embedding dimension of $X$ equals 4. 
\item [{\bf (c)}] $(X , SL_2(\C) , \rho)$ is absolutely nonextendable. 
\end{itemize}
\end{theorem}
The presentation $\varphi$ in part (a) of this theorem is given in {\it Section\,\ref{SL2-nonextend}}. 

Danielewski \cite{Danielewski.89} showed that $Y$ is noncancelative, that is, there exists a surface $\tilde{Y}$ with $Y\not\cong\tilde{Y}$ and 
$Y\times\C^1\cong\tilde{Y}\times\C^1$. 
In {\it Section\,\ref{cylinder}} we use the structure of ${\rm Aut}_{\C}(X)$ to show that $X$ is also noncancelative. 
We also settle the long-standing question of whether there exist inequivalent closed embeddings of $Y$ in $\A_k^3$; 
see \cite{Freudenburg.Moser-Jauslin.03}, Question 2. 
The following theorem is proved in {\it Section\,\ref{embedding}}.
\begin{theorem}\label{two-embed} 
Define regular functions $p,F,r,G\in k[x,y,z]=\mathcal{O}(\A_k^3)$ by
\[
p=xz-y^2\,\, ,\,\, F=xz-y^3 \,\,, \,\, r=yF-x \,\, ,\,\, G=x^{-1}(F^4+r^3)\, .
\]
\begin{itemize}
\item [{\bf (a)}] Every fiber of $p$ and $G$ other than the zero fiber is isomorphic to $Y=SL_2(k)/T$. 
\item [{\bf (b)}] Every fiber of $p$ is normal.
\item[{\bf (c)}] The zero fiber of $G$ is not normal. 
\end{itemize}
Consequently, there is no algebraic automorphism of $\A_k^3$ 
which carries the hypersurface $G=1$ to the hypersurface $p=1$.
\end{theorem}

The problem of finding a minimal presentation for a group action on an affine variety 
generalizes the problem of determining whether a group action on affine space is linearizable. 
For example, Masuda and Petrie \cite{Masuda.Petrie.94} consider $O_2$-modules of dimension 5 defined by
\[
^{\lambda}(a,b,x,y,z)=(\lambda^{-1}a,\lambda^1b,\lambda^{-n}x,\lambda^ny,z)\,\,\ ,\,\, \lambda\in\C^*\quad\text{and}\quad \tau (a,b,x,y)=(b,a,y,x,z) 
\]
where $n\in\Z$, $n\ge 1$. Given a univariate polynomial $f(t)$ over $\C$ with $f(0)\ne 0$, 
the principal ideal $I_n=(b^nx+a^ny-f(ab)z)$ is $O_2$-invariant and defines a subvariety $V_n\subset\C^5$ with $V_n\cong \C^4$. 
This induces an $O_2$-action on $V_n$ with presentation $\varphi :\C [a,b,x,y,z]\to \C [a,b,x,y,z]/I_n$. 
The authors prove that, if $n\ge 2$, then it is possible to choose $f(t)$ so that this presentation is minimal, equivalently, the $O_2$-action on $\C^4$ is nonlinearizable. 

\setcounter{tocdepth}{1}
\tableofcontents


\section{Preliminaries}\label{prelim}
Throughout, $k$ is a field of characteristic zero and $\bar{k}$ is its algebraic closure. 
A {\bf $k$-domain} $B$ is an integral domain containing $k$. Given $f\in B$, $B_f$ denotes the localization $B_f=B[f^{-1}]$.
Given $n\in\N$, $B^{[n]}$ is the polynomial ring in $n$ variables over $B$.

\subsection{Locally nilpotent derivations} Let $B$ be an affine $k$-domain. The set of $k$-derivations of $B$ is ${\rm Der}_k(B)$. 
Given $D\in {\rm Der}_k(B)$, the kernel of $D$ is denoted $\krn D$. $D$ is {\bf locally nilpotent} if, given $b\in B$, $D^nb=0$ for $n\gg 0$. 
The set of locally nilpotent derivations of $B$ is ${\rm LND}(B)$. 
This set is in bijective correspondence with the set of $\G_a$-actions on ${\rm Spec}(B)$ via the exponential map, where $\G_a$ is the additive group of $k$. 
We say that $D\in {\rm LND}(B)$ is {\bf irreducible} if the image $DB$ is contained in no proper principal ideal of $B$. 
The Makar-Limanov invariant of $B$ is $ML(B)=\bigcap_{D\in{\rm LND}(B)}\krn D$, and if $X={\rm Spec}(B)$, then $ML(X)=ML(B)$. 

Given $D\in {\rm LND}(B)$, set $A=\krn D$. An element $s\in B$ is a {\bf local slice} for $D$ if $Ds\ne 0$ and $D^2s=0$. 
A local slice $s$ for $D$ is a {\bf slice} if $Ds=1$. The {\bf Slice Theorem} says that, for any local slice $s$, $B_{Ds}=A_{Ds}[s]\cong_kA_{Ds}^{[1]}$. 
The degree function $\deg_D$ on $B$ is the restriction of $\deg_s$ on $B_{Ds}$. It does not depend on the choice of local slice $s$. 

Given a positive integer $n$ and $f,g\in B$, the $n^{\text{th}}$ {\bf transvectant} of $f$ and $g$ is:
\[
[f,g]_n^D=\sum_{i=0}^n(-1)^{n-i}D^ifD^{n-i}g
\]
If $\deg_D(f)\le n$ and $\deg_D(g)\le n$ then $[f,g]_n^D\in A$. This is an important way to construct elements of the kernel of $D$. See \cite{Freudenburg.17}, 2.11.1. 

In \cite{Bandman.Makar-Limanov.01}, Bandman and Makar-Limanov study the class of all smooth affine surfaces $S$ over $\C$ with $ML(S)=\C$.
They show the following.
\begin{theorem}\label{Band-ML} Let $S$ be a smooth affine surface over $\C$ with $ML(S)=\C$. 
The following conditions are equivalent.
\begin{enumerate}[label=(\roman*)]
\item $S$ is isomorphic to a hypersurface in $\A_{\C}^3$. 
\item $S$ is isomorphic to $\{ (x,y,t)\in\A_k^3\, |\, xy=p(t)\}$ for some $p\in k[t]$ with simple roots.
\item $S$ admits a fixed-point free $\G_a$-action with reduced fibers.
\end{enumerate}
\end{theorem}
The equivalence of conditions (i) and (iii) is gotten by combining Theorem 1 and Theorem 2 in their paper. The equivalence of conditions (i) and (ii) is found in the proof of Lemma 5. 

See \cite{Freudenburg.17} for further details on locally nilpotent derivations. 

\subsection{Fundamental pairs of LNDs} Suppose that $B$ is an affine $k$-domain. We can view ${\rm Der}_k(B)$ as a Lie algebra with the usual bracket operation. 
The pair $(D,U)\in {\rm LND}(B)^2$ is a {\bf fundamental pair} for $B$ if $D$ and $U$ satisfy the relations
\[
[D,[D,U]]=-2D \quad\text{and}\quad [U,[D,U]]=2U
\]
i.e., $D$ and $U$ generate a Lie subalgebra isomorphic to $\mathfrak{sl}_2(k)$. 
Fundamental pairs for $B$ are in bijective correspondence with algebraic actions of $SL_2(k)$ on $B$; see \cite{Freudenburg.24}. 
In particular, an $SL_2(k)$-action on $B$ induces a fundamental pair $(D,U)$ for $B$, since the upper and lower unipotent triangular subgroups of $SL_2(k)$, 
each being isomorphic to $\G_a$, give rise to locally nilpotent derivations of $B$ which satisfy these relations. 

Suppose that $(D,U)$ is a fundamental pair for $B$. Set $E=[D,U]$. 
By \cite{Andrist.Draisma.Freudenburg.Huang.Kutzschebauch.ppt}, Theorem 2.6, $E$ is a semi-simple derivation which induces a $\Z$-grading of $B$ given by 
$B=\bigoplus_{d\in\Z}B_d$ where $B_d=\krn (E-dI)$. The kernel $A=\krn D$ is a graded affine subalgebra and $A=\bigoplus_{d\in\N}A_d$ is $\N$-graded.   
In addition:
\[
A_0=\krn D\cap\krn U=B^{SL_2(k)}
\]
Given $f\in A_d$, $\widehat{f}$ denotes the vector space in $B$ with basis $\{ U^jf\, |\, 0\le j\le d\}$. Note that $U(U^if)=U^{i+1}f$, so $U$ restricts to $k[\widehat{f}]$. In addition, $D(U^if)=i(d-i+1)U^{i-1}f$ so $D$ restricts to $k[\widehat{f}]$; see \cite{Freudenburg.24}, Lemma 3.5. 
Therefore, $k[\widehat{f}]$ is a $(D,U)$-invariant subalgebra, and is the smallest invariant subalgebra of $B$ containing $f$. 

Suppose that $B'$ is another $k$-domain endowed with the fundamental pair $(D',U')$. 
A $k$-algebra morphism $\varphi :B'\to B$ is {\bf equivariant} if $D\varphi =\varphi D'$ and $U\varphi =\varphi U'$. 

Given $n\in\N$, let $\V_n=kX_0\oplus\cdots\oplus kX_n$ be the irreducible $SL_2$-module of dimension $n+1$ over $k$. 
Then $\V_n$ induces the fundamental pair $(D_n,U_n)$ for the polynomial ring $k[\V_n]=k[X_0,\hdots X_n]$ defined by
\[
D_n(X_i)=X_{i-1} \,\, (1\le i\le n) \,\, ,\,\, D_n(X_0)=0
\]
and:
\[
U_n(X_i)=(i+1)(n-i)X_{i+1}\,\, (0\le i\le n-1) \,\, ,\,\, U(X_n)=0
\]
For the induced degree function on $k[\V_n]$, each $X_i$ is homogeneous and $\deg X_i=n-2i$, $0\le i\le n$. 
\begin{lemma}\label{presentation} {\rm (see \cite{Andrist.Draisma.Freudenburg.Huang.Kutzschebauch.ppt})} 
Let $B$ be an affine $k$-domain with fundamental pair $(D,U)$ and $A=\krn D$. 
Let $A=\bigoplus_{d\in\N}A_d$ be the $\N$-grading induced by $(D,U)$, and 
let $A=k[f_1,\hdots ,f_r]$ where $f_i\in A_{d_i}$, $1\le i\le r$.
\begin{itemize}
\item [{\bf (a)}] $B=k[\widehat{f_1},\hdots ,\widehat{f_r}]$
\item [{\bf (b)}]  
There exists an equivariant surjection $\varphi :k[\V_{d_1}\oplus\cdots\oplus\V_{d_r}]\to B$ with
$\varphi (\V_{d_i})=\widehat{f_i}$.
\end{itemize} 
\end{lemma} 
Part (b) of this lemma thus gives a specific closed embedding of ${\rm Spec}(B)$ into an $SL_2(k)$-module of dimension
$r+d_1+\cdots +d_r$, gotten by calculating a set of homogeneous generators of $A$. 

Let $n\ge 1$ be given and write $n=2m+\epsilon$ where $\epsilon\in\{ 0,1\}$. 
It is well-known that the quadratic invariants for $\V_n$ are:
\[
T_{2i}=\sum_{r=0}^{2i}(-1)^rX_rX_{2i-r}\,\, , \,\, 0\le i\le m
\]
The preceding lemma implies the following. 
\begin{corollary}\label{quadratics}
The vector space of quadratic forms in $k[\V_n]$ decomposes into $SL_2$-invariant
subspaces $\widehat{T}_0\oplus\widehat{T}_2\oplus\cdots\oplus\widehat{T}_{2m}$.
\end{corollary}

See \cite{Freudenburg.24, Andrist.Draisma.Freudenburg.Huang.Kutzschebauch.ppt} for further details on fundamental pairs. 

\subsection{Finite subgroups of $SL_2(\C )$ and their binary invariants}\label{finite}

\subsubsection{Representations} The finite subgroups of $SL_2(\C )$ were given by Felix Klein in \cite{Klein.1884}.
The list is comprised of two countably infinite families and three sporadic groups:  the cyclic groups $\Z_n$ and binary dihedral groups ${\rm BD}_{4n}$ $(n\ge 2)$, the binary tetrahedral group ${\rm BT}_{24}$, the binary octahedral group ${\rm BO}_{48}$ and the binary icosahedral group ${\rm BI}_{120}$. 
The special unitary group $SU(2)\subset SL_2(\C )$ gives a representation of the group of unit quaternions: 
\[
\begin{pmatrix} z&-w\cr \bar{w}&\bar{z}\end{pmatrix}\,\, ,\,\, \vert z\vert^2+\vert w\vert^2=1
\]
$SU(2)$ is generated by rotations in 3-space (up to sign), meaning that $SU(2)$ contains the symmetry group of each of the platonic solids studied by Klein. 
Moreover, any finite subgroup of $SL_2(\C )$ can be conjugated to $SU(2)$, so classifying finite subgroups of $SL_2(\C )$ reduces to classifying those in $SU(2)$. 
The representations of finite groups and corresponding invariants given below are those derived by Klein. 
Isomorphic finite subgroups of $SL_2(\C )$ are known to be conjugate, so a single representation of each is sufficient. 
Moreover, in these results, the field $\C$ can be replaced by any algebraically closed field $k$ of characteristic zero. 

Given $n\ge 2$, let $\zeta_n\in\C$ be an element of order $n$. Define 
\[
\psi_n  = \begin{pmatrix} \zeta_n & 0\cr 0 & \zeta_n^{-1}\end{pmatrix} 
\]
and elements:
\[
\vec{i} = \begin{pmatrix} i& 0\cr 0 & -i \end{pmatrix} \quad
\vec{j} = \begin{pmatrix} 0 & 1\cr -1 & 0 \end{pmatrix} \quad 
\vec{k} =  \begin{pmatrix} 0& i\cr i & 0 \end{pmatrix} \quad
\sigma = {\textstyle\frac{1}{2}}\left(I+\vec{i}+\vec{j}+\vec{k}\right)
\]
Then $\vec{i},\vec{j},\vec{k}$ are elements of order 4 which generate the quaternion group ${\rm BD}_8$, $\vec{j}$ normalizes $\langle\psi_n\rangle$,
and $\sigma$ is an element of order 6 which normalizes ${\rm BD}_8$.
Define also
\[
\kappa = {\textstyle\frac{1}{\sqrt{5}}}\begin{pmatrix} \zeta_5^4-\zeta_5 & \zeta_5^2-\zeta_5^3\cr \zeta_5^2-\zeta_5^3&\zeta_5-\zeta_5^4\end{pmatrix}
\]
which is of order 4. We have
\begin{enumerate}
\item $\Z_n=\langle \psi_n \rangle$ \qquad $(n\ge 2)$
\smallskip
\item ${\rm BD}_{4n}=\langle \Z_{2n} ,\vec{j}\,\,\rangle$ \quad $(n\ge 2)$
\smallskip
\item ${\rm BT}_{24}=\langle {\rm BD}_8 ,\sigma^2 \,\,\rangle=\langle \sigma , \vec{j}\, |\, \sigma^3=(\sigma^{-1}\vec{j}\, )^3=\vec{j}^2=-I \,\,\rangle \cong {\rm SL (2,3)}$
\smallskip
\item ${\rm BO}_{48}=\langle {\rm BT}_{24},\psi_8\,\, \rangle=\langle \sigma , \psi_8\, |\, \sigma^3=\psi_8^4=(\sigma\psi_8)^2=-I\,\,\rangle $ 
\smallskip
\item ${\rm BI}_{120}=\langle \psi_5,\kappa ,\vec{j}\rangle \cong {\rm SL (2,5)}$ 
\end{enumerate}
where $\psi_8$ normalizes ${\rm BT}_{24}$. The only proper normal subgroup of ${\rm BI}_{120}$ is the center $\{\pm I\}$.\footnote{
A simpler representation of ${\rm BI}_{120}$ is given by $\langle \sigma ,\rho\, |\, \rho^5=\sigma^3=(\sigma\rho)^2=-I\,\,\rangle$ where 
$ \rho = \frac{1}{2}\left(v I+w\vec{i}+\vec{j}\right)$ for $v=\frac{1+\sqrt{5}}{2}$ and $w=v^{-1}=\frac{-1+\sqrt{5}}{2}$. }

\subsubsection{Binary invariant rings} Each finite group $H\subset SL_2(\C )$ acts on the polynomial ring $\C [x,y]$ and the invariant subalgebra
is 3-generated: $\C [x,y]^H=\C [r,s,t]$. The prime relation of the generators defines a rational Du Val singularity of type $A_n, D_n, E_6,E_7,E_8$, respectively, as follows.

\begin{align*}
\text{cyclic group} \,\,\Z_n\,\,  (n\ge 2) && A_n && r^2+s^2+t^n \\
\text{binary dihedral group} \,\, {\rm BD}_{4n} \,\, (n\ge 2) && D_n  && tr^2+s^2+t^{n+1}  \\
\text{binary tetrahedral group} \,\, {\rm BT}_{24} && E_6 && r^4+s^3+t^2 \\
\text{binary octahedral group} \,\, {\rm BO}_{48} && E_7 && r^3s+s^3+t^2 \\
\text{binary icosahedral group} \,\, {\rm BI}_{120} && E_8 && r^5+s^3+t^2 
\end{align*}
Specific generators for the representations given above are given as follows. 

\begin{enumerate}
\item $\C [x,y]^{\Z_n}=\C [r_n,s_n,t]=\C [x^n+y^n,x^n-y^n,xy]$ where $r_n^2+s_n^2=t^n$.
\smallskip
\item $\C [x,y]^{\text{BD}_{4n}}=\C [r_{2n},ts_{2n},t^2]$ where $(t^2)r_{2n}^2-(ts_{2n})^2=4(t^2)^{n+1}$. 
 \smallskip
\item $\C [x,y]^{\text{BT}_{24}}=\C [v_1,v_2,v_3]$ where $v_1^3+v_3^4=v_2^2$ for polynomials:
\begin{eqnarray*}
v_1 &=& (x^8+y^8)+14x^4y^4 \\
v_2 &=& (x^{12}+y^{12})-33(x^8y^4+x^4y^8) \\
v_3 &=& xy(x^4-y^4) 
\end{eqnarray*}
\item $\C [x,y]^{\text{BO}_{48}}=\C [v_1,v_2v_3,v_3^2]$ where $(v_3^2)((v_3^2)^2+v_1^3)=(v_2v_3)^2$. 
\smallskip
\item $\C [x,y]^{\text{BI}_{120}}=\C [w_1,w_2,w_3]$ where $1728\,w_1^5+w_2^3=w_3^2$ for polynomials:
\begin{eqnarray*}
w_1 &=& xy(x^{10}+11x^5y^5-y^{10}) \\
w_2 &=& (x^{20}+y^{20})-228(x^{15}y^5-x^5y^{15})+494x^{10}y^{10} \\
w_3 &=& (x^{30}+y^{30})+522(x^{25}y^5-x^5y^{25})-10005(x^{20}y^{10}+x^{10}y^{20})
\end{eqnarray*}
\end{enumerate}


\section{Minimal presentations of $SL_2$-homogeneous threefolds}\label{present}

In this section, assume that $k$ is algebraically closed. 

\subsection{Left and right actions}
The actions of $SL_2(k)$ on itself by left and right multiplication, respectively, commute. If we restrict the right action to a closed reductive subgroup $H\subset SL_2(k)$, then we get an $SL_2$-action on the quotient $SL_2(k)/H$ induced by left multiplication. In other words, the right action determines the variety, and the left action determines the $SL_2$-action on that variety.

Let $B=k[x,y,z,w]$ where $xw-yz=1$, the coordinate ring of $SL_2(k)$. 
The fundamental pair $(\delta ,\upsilon )$ which defines the left action is given by:
\[
\delta\begin{pmatrix} x& y\cr z&w \end{pmatrix}=
\begin{pmatrix} 0&0\cr 1&0\end{pmatrix}\begin{pmatrix} x& y\cr z&w \end{pmatrix}
= \begin{pmatrix} 0&0\cr x&y\end{pmatrix} \quad\text{and}\quad 
\upsilon\begin{pmatrix} x& y\cr z&w \end{pmatrix}=\begin{pmatrix} 0&1\cr 0&0\end{pmatrix}\begin{pmatrix} x& y\cr z&w \end{pmatrix} = 
\begin{pmatrix} z&w\cr 0&0\end{pmatrix}
\]
Set $A=\krn\delta =k[x,y]$.
\subsection{Generators of coordinate rings}
Given $n\in\N$ consider the irreducible representation $\V_n$ of $SL_2(k)$.
Let $k[\V_n]=k[X_0,\hdots ,X_n]$ with induced fundamental pair $(D_n,U_n)$, as above. 
In addition to the quadratic polynomials $T_{2i}$ we also need the cubic element of $\krn D_n$ given by:
\[
G=3X_0^2X_3-3X_0X_1X_2+X_1^3 \,\, (n\ge 3)
\]
Given the integer $n\ge 2$, 
let $\varphi_n:k[\V_n]\to B$ be the equivariant $k$-algebra homomorphism defined by $\varphi_n(X_n)=z^n+w^n$. 
\begin{lemma}\label{images} 
For each integer $1\le i\le \frac{1}{2}n$ :
\[
\varphi_n(T_{2i})=\displaystyle 2\frac{(n!)^2}{(2i)!}\, (xy)^{n-2i}
\]
\end{lemma}
Observe that, for $i=0$, we have $\varphi_n(T_0)=\varphi_n(X_0^2)=(n!)(x^n+y^n)$. 

\begin{proof} Recall that, for any positive integer $b$ :
\begin{equation}\label{factorialsum}
\sum_{a=0}^b(-1)^a{b\choose a}=0
\end{equation}
By equivariance, given $0\le j\le n$ we have:
\[
\varphi_n(X_j)=\varphi_nD^{n-j}(X_n)=\delta^{n-j}\varphi_n(X_n)=\delta^{n-j}(z^n+w^n)=(n!/j!)(x^{n-j}z^j+y^{n-j}w^j)
\]
Given $1\le i\le \frac{1}{2}n$ we have:
\begin{eqnarray*}
\varphi_n(T_{2i}) &=& \sum_{j=0}^{2i}(-1)^j\,\frac{n!}{j!}\frac{n!}{(2i-j)!}(x^{n-j}z^j+y^{n-j}w^j)(x^{n-2i+j}z^{2i-j}+y^{n-2i+j}w^{2i-j}) \\
&=&\frac{(n!)^2}{(2i)!}\sum_{j=0}^{2i}(-1)^j{2i\choose j}\left( x^{2n-2i}z^{2i}+y^{2n-2i}w^{2i}+x^{n-j}y^{n-2i+j}z^jw^{2i-j}+x^{n-2i+j}y^{n-j}z^{2i-j}w^j\right) 
\end{eqnarray*}
Since the terms $x^{2n-2i}z^{2i}+y^{2n-2i}w^{2i}$ do not involve $j$, equation (\ref{factorialsum}) implies that these terms cancel. Therefore:
\begin{eqnarray*}
\varphi_n(T_{2i}) 
&=& \frac{(n!)^2}{(2i)!}\sum_{j=0}^{2i}(-1)^j{2i\choose j}\left( x^{n-j}y^{n-2i+j}z^jw^{2i-j}+x^{n-2i+j}y^{n-j}z^{2i-j}w^j\right) \\
&=& \frac{(n!)^2}{(2i)!}\sum_{j=0}^{2i}(-1)^j{2i\choose j}\left( x^{(n-2i)+(2i-j)}y^{(n-2i)+j}z^jw^{2i-j}+x^{(n-2i)+j}y^{(n-2i)+(2i-j)}z^{2i-j}w^j\right) \\
&=&  \frac{(n!)^2}{(2i)!}\, (xy)^{n-2i}\sum_{j=0}^{2i}(-1)^j{2i\choose j}\left( x^{2i-j}y^jz^jw^{2i-j}+x^jy^{2i-j}z^{2i-j}w^j\right) \\
&=&  \frac{(n!)^2}{(2i)!}\, (xy)^{n-2i}\left( (yz-xw)^{2i}+(xw-yz)^{2i}\right) \\
&=&  2\frac{(n!)^2}{(2i)!}\, (xy)^{n-2i}
\end{eqnarray*}
\end{proof}

\begin{lemma}\label{alg-independent} 
Given odd $m\ge 3$, the quadratic polynomial $T_{m+1}\in k[\V_{2m}]$ 
has degree $2(m-1)$ and 
\[
T_{m+1},UT_{m+1},\hdots ,U^{2(m-1)}T_{m+1}
\]
are algebraically independent over $k$. 
\end{lemma} 

\begin{proof} 
For $0\le j\le m-1$, $U^jT_{m+1}$ supports the monomial $X_0X_{m+1+j}$ but supports no monomial divisible by $X_{2m}$. By degree considerations, $U^jT_{m+1}$ supports no other monomial divisible by $X_0$. 

For $j=m-1$, $U^{m-1}T_{m+1}$ supports the monomial $X_0X_{2m}$ but supports no other monomial divisible by $X_0$ or $X_{2m}$. 

For $m\le j\le 2(m-1)$, $U^jT_{m+1}$ supports the monomial $X_{j-m+1}X_{2m}$ but supports no monomial divisible by $X_0$. By degree considerations, $U^jT_{m+1}$ supports no other monomial divisible by $X_{2m}$. 

Therefore, the $(2m-1)\times (2m+1)$ jacobian matrix 
\[  
\mathcal{J}= \frac{\partial (T_{m+1},UT_{m+1},\hdots ,U^{2(m-1)}T_{m+1})}{\partial (X_0,X_1,\hdots ,X_{2m})}  
\]
evaluated at $X_1=\cdots =X_{2m-1}=0$ has exactly one nonzero entry in each column except for the middle ($m^{\text{th}}$) column, which is a column of zeros; 
and has exactly one nonzero entry in each row except for the middle ($m^{\text{th}}$) row, whose first and last entries are nonzero. 
This implies that the rank of $\mathcal{J}(X_0,0,\hdots ,0,X_{2m})$ equals $2m-1$, so the rank of $\mathcal{J}$ equals $2m-1$. 
It follows that $T_{m+1},UT_{m+1},\hdots ,U^{2(m-1)}T_{m+1}$ are algebraically independent. 
See \cite{Freudenburg.17}, 3.2.3. 
\end{proof}

\begin{theorem}\label{threefolds} The following equalities hold. 
\smallskip
\begin{itemize}
\item [{\bf (a)}] $B^{\Z_n}=k[\widehat{x^n+y^n},\widehat{xy}]$ for all $n\ge 2$. 
\medskip
\item[{\bf (b)}] $B^{\Z_n}=k[\widehat{x^n+y^n}]$ for all odd $n\ge 3$.
\medskip
\item[{\bf (c)}] $B^{\Z_{2n}}=k[\widehat{f_n}]$ for all odd $n\ge 3$, where
$f_n=x^{2n}+(xy)^n+y^{2n}$. 
\medskip
\item [{\bf (d)}] $B^{{\rm BD}_{4n}}=k[\widehat{r_{2n}}]$ for all $n\ge 2$, where $r_{2n}=x^{2n}+y^{2n}$.
\medskip
\item [{\bf (e)}] $B^{{\rm BT}_{24}}=k[\widehat{v_3}]$ for $v_3=xy(x^4-y^4)$.
\medskip
\item [{\bf (f)}] $B^{{\rm BO}_{48}}=k[\widehat{v_1}]$ for $v_1=x^8+14x^4y^4+y^8$.
\medskip
\item [{\bf (g)}] $B^{{\rm BI}_{120}}=k[\widehat{w_1}]$ for $w_1=xy(x^{10}+11x^5y^5-y^{10})$.
Moreover, $B^{{\rm BI}_{120}}$ is a UFD.
\end{itemize}
\end{theorem}

\begin{proof} 
Part (a). 
For any integer $n\ge 2$ we have:
\[
[\upsilon (x^n+y^n) , \upsilon (xy)]_1^{\delta}=n(x^n-y^n)
\]
Therefore:
\[
x^n-y^n\in  k[\widehat{x^n+y^n},\widehat{xy}]
\implies \widehat{x^n-y^n}\subset k[\widehat{x^n+y^n},\widehat{xy}]
\implies B^{\Z_n}=k[\widehat{x^n+y^n},\widehat{xy}]
\]
Part (b). From part (a), 
$B^{\Z_n}=k[\widehat{x^n+y^n},\widehat{xy}]$, and from {\it Lemma\,\ref{images}}, $xy\in k[\widehat{x^n+y^n}]$. 
Therefore:
\[
\widehat{xy}\subset k[\widehat{x^n+y^n}]
\implies B^{\Z_n}=k[\widehat{x^n+y^n}]
\]
Part (c). By part (a), $B':=B^{\Z_n}=k[\widehat{z^{2n}+y^{2n}},\widehat{xy}]$. 
Let $A'=A\cap B'=k[x^n,xy,y^n]$. Define
$B''=k[\widehat{f_n}]\subseteq B'$ and $A''=A'\cap B''$.
We claim that $xy\in A''$. This claim, once verified, implies $x^{2n}+y^{2n}\in A''$ as well, and $B^{\Z_{2n}}=B''$. 

In order to prove the claim, let $(X_i)$ be coordinates for $\V_{2n}$ and define the equivariant surjection:
\[
\varphi :k[\V_{2n}]\to B''\,\, ,\,\, \varphi (X_{2n})=f_n=z^{2n}+(zw)^n+w^{2n}
\]
By direct computation we find:
\[
p:=[\upsilon^2f_n,\upsilon^2f_n]_2^{\delta}
=-4n(2n-1)(xy)^{n-2}\left( 2(n-1)(x^{2n}+y^{2n})+(8n-5)(xy)^n\right)\in A''
\]
Since $k[f_n,p]\cong k^{[2]}$ it follows that $\dim A''=2$, so $\dim B''=3$. 

We have $\deg T_{n+1}=2(n-1)$ and 
$\varphi (T_{n+1})\in A^{\prime}_{2(n-1)}=k\cdot (xy)^{n-1}$.
If $T_{n+1}\in\krn\varphi$ then $\widehat{T}_{n+1}\subset\krn\varphi$.
By {\it Lemma\,\ref{alg-independent}}, $T_{n+1},UT_{n+1},\hdots ,U^{2(n-1)}T_{n+1}$ are algebraically independent. Since $k[\V_{2n}]\cong k^{[2n+1]}$ it follows that
\[
\dim B''=\dim k[\V_{2n}]/\krn\varphi \le (2n+1)-(2n-1)=2
\]
which gives a contradiction. Therefore, $\varphi (T_{n+1})=c(xy)^{n-1}$ for some $c\in k^*$ and
$(xy)^{n-1}\in A''$. 

By direct computation, we find:
\begin{enumerate}[label=(\roman*)]
\item $[\upsilon (f_n),\upsilon ((xy)^{n-1})]_1^{\delta}=8n(2n-1)(2n-3)(xy)^{n-2}(x^{2n}-y^{2n})\in A''$.
\medskip
\item 
$[\upsilon^2(f_n),\upsilon^2((xy)^{n-1})]_2^{\delta}
=4n(xy)^{n-3}\left( (2n^2-5n+2)(x^{2n}+y^{2n})-2(n-1)(xy)^n\right)\in A''$
\medskip
\item $[\upsilon ((xy)^{n-2}(x^{2n}-y^{2n})),\upsilon ((xy)^{n-1}]_1^{\delta}
=(n-1)(xy)^{2n-4}(x^{2n}+y^{2n})\in A''$ 
\end{enumerate}
It follows that:
\begin{eqnarray*}
&&(xy)^{n-1}(xy)^{n-3}((2n^2-5n+2)(x^{2n}+y^{2n})-2(n-1)(xy)^n)\\
&=&(xy)^{2n-4}((2n^2-5n+2)(x^{2n}+y^{2n})-2(n-1)(xy)^n)\in A''
\end{eqnarray*}
Subtracting from this $(2n^2-5n+2)(xy)^{2n-4}(x^{2n}+y^{2n})$ shows 
$(xy)^{3n-4}\in A''$.
From this we find
\[
(xy)^{3n-4}((xy)^{n-1})^{-3}=(xy)^{-1}\in A''_{(xy)^{n-1}}
\implies
xy\in A''_{(xy)^{n-1}}
\]
where $A''_{(xy)^{n-1}}$ is the localization of $A''$ at $(xy)^{n-1}$. Since $xy$ is a unit of $A''_{(xy)^{n-1}}$ and the units of $A''$ are $k^*$, it follows that $xy=t^m$ for some divisor $t$ of $(xy)^{n-1}$ in $A''$ and $m\in\Z$. Up to scalar multiples, the only divisors of $(xy)^{n-1}$ in $A=k[x,y]$ are $x^ay^b$, $0\le a,b\le n-1$. 
So $xy=x^{am}y^{bm}$ in $A$, which implies 
$a=b=m=1$. Therefore, $xy\in A''$ and the claim is proved.

Part (d). By part (a) we have $B^{\text{BD}_{4n}}\subset B^{\Z_{2n}}=k[\widehat{x^{2n}+y^{2n}},\widehat{xy}]$. Under the action of $\Z_2=\langle\vec{j}\rangle$ we find that:
\[
B^{\text{BD}_{4n}}=\left( B^{\Z_{2n}}\right)^{\vec{j}}=k[\widehat{x^{2n}+y^{2n}},\widehat{xy}]^{\Z_2}=k[\widehat{x^{2n}+y^{2n}},\widehat{(xy)^2}]
\]
By {\it Lemma\,\ref{images}}, $\varphi_{2n}(T_{2(n-1)})=\zeta (xy)^2$ for some nonzero $\zeta\in\N$, which implies $(xy)^2\in k[\widehat{x^{2n}+y^{2n}}]$. 
This proves part (d).

Part (e). Let $(X_i)$ be coordinates on $\V_6$ and define the equivariant morphism $\psi :k[\V_6] \to B$ by $\psi (X_6)=zw(z^4-w^4)$, noting that the image of $\psi$ is $k[\widehat{v_3}]$. 
The reader can check that
\[
\psi (T_2)=\psi (2X_0X_2-X_1^2)=-4v_1 \quad \text{and}\quad \psi (G)=\psi (3X_0^2X_3-3X_0X_1X_2+X_1^3)=8v_2\, .
\]
Therefore, $v_1,v_2\in k[\widehat{v_3}]$ implies $\widehat{v_1},\widehat{v_2}\subset k[\widehat{v_3}]$ and $B^{{\rm BT}_{24}}=k[\widehat{v_3}]$.
This proves part (d).

Part (f). Let $(X_i)$ be coordinates on $\V_8$ and define the equivariant morphism $\rho :k[\V_8] \to B$ by $\rho (X_8)=z^8+14z^4w^4+w^8$, noting that the image of $\rho$ is $k[\widehat{v_1}]$. 
The reader can check that
\[
\rho (T_2)=\rho (2X_0X_2-X_1^2)=108v_3^2 \quad \text{and}\quad \rho (G)=\rho (3X_0^2X_3-3X_0X_1X_2+X_1^3)=108v_2v_3\, .
\]
Therefore, $v_3^2,v_2v_3\in k[\widehat{v_1}]$ implies $\widehat{v_2^2},\widehat{v_2v_3}\subset k[\widehat{v_1}]$ and $B^{{\rm BO}_{48}}=k[\widehat{v_1}]$.
This proves part (e).

Part (g). 
The reader can check that
\[
[\upsilon (w_2),\upsilon (w_1)]_1^{\delta}=20w_3
\quad \text{and} \quad 
[\upsilon^2(w_1),\upsilon^2(w_1)]_2^{\delta}=-484w_2
\]
Therefore, $w_2,w_3\in k[\widehat{w_1}]$ implies $\widehat{w_2},\widehat{w_3}\subset k[\widehat{w_1}]$ and $B^{{\rm BI}_{120}}=k[\widehat{w_1}]$.

Let $H={\rm BI}_{120}$. Since $B^*=k^*$, we have $(B^H)^*=k^*$. Let $\sigma\in {\rm Hom}(H,k^*)$ be given. 
Since $k^*$ is abelian, $\sigma$ factors through the quotient $H/[H,H]$ where $[H,H]$ is the commutator subgroup. It is well-known that $H$ is a perfect group, that is, $H=[H,H]$. 
Therefore, $\sigma$ is trivial. By {\it Lemma\,\ref{Samuel}} below, $B^H$ is a UFD. This proves part (f). 
\end{proof}

\begin{lemma}\label{Samuel} {\rm (\cite{Samuel.68}, Section 5, Part (3))} Let $R$ be a UFD over a field $K$ and let the finite group $G$ act on $R$ by $K$-algebra automorphisms. 
If $R^*=K^*$ and ${\rm Hom}(G,K^*)$ is trivial then $R^G$ is a UFD.
\end{lemma} 

\begin{remark} {\rm The proof for part (c) of {\it Theorem\,\ref{threefolds}} is rather {\it ad hoc}. The key is to find a preimage of $xy$. The cubic polynomial}
\[
\big[ U^{2(n-1)}T_{n+1}\, ,\, X_{2(n-1)}\big]_{2(n-1)}^{D_n}
\]
{\rm is of degree two, so its image is $c_n(xy)$ for $c_n\in k$. If $c_n\ne 0$ for all odd $n$, this gives a simpler proof. 
We checked that $c_3\ne 0$ and $c_5\ne 0$, but could not show $c_n\ne 0$ for all odd $n$.}
\end{remark}

\subsection{Minimal presentations} 
{\it Theorem\,\ref{threefolds}} yields the following presentations. 
Let $n\ge 2$. Coordinates on $\V_n$ are $(X_i)$, $0\le i\le n$, 
and coordinates on $\V_n\oplus \V_2$ are $(X_i,Y_j)$, $0\le i\le n$, $0\le j\le 2$. 
\begin{enumerate}[label=(\roman*)]
\item For odd $n\ge 3$: $\gamma_n:k[\V_n]\to B^{\Z_n}$ where $\gamma_n(X_n)=z^n+w^n$
\smallskip
\item For $n\equiv 0\pmod 4$ and for $n=2$:
\[
\alpha_n:k[\V_n\oplus \V_2]\to B^{\Z_n} \,\, ,\,\, \alpha_n(X_n)=z^n+w^n \,\, ,\,\,  \alpha_n(Y_2)=zw
\]
\item For $n\equiv 2\pmod 4$ with $n=2m\ge 6$:
\[
\beta_n:k[\V_n]\to B^{\Z_n}\,\, ,\,\, \beta_n(X_n)=z^{2m}+(zw)^m+w^{2m}
\]
\item $\delta_n :k[\V_{2n}]\to B^{{\rm BD}_{4n}}$ where $\delta_n(X_{2n})=z^{2n}+w^{2n}$
\smallskip
\item $\epsilon :k[\V_6]\to B^{{\rm BT}_{24}}$ where $\epsilon (X_6)=zw(z^4-w^4)$
\smallskip
\item $\eta :k[\V_8]\to B^{{\rm BO}_{48}}$ where $\eta (X_8)=z^8+14z^4w^4+w^8$
\smallskip
\item $\kappa :k[\V_{12}]\to B^{{\rm BI}_{120}}$ where $\kappa (X_{12})=zw(z^{10}+11z^5w^5-w^{10})$
\end{enumerate}

\begin{example}\label{V2+V2} {\rm $H=\Z_2=\{ \pm I\}$ is the center of $SL_2(k)$ and $B^H=k[\widehat{x^2+y^2},\widehat{xy}]$ has six generators. 
Coordinates for $\V_2\oplus \V_2$ are $((X_0,X_1,X_2), (Y_0,Y_1,Y_2))$. 
The equivariant surjection $\alpha_2 :k[\V_2\oplus \V_2]\to B^H$ is defined by $\alpha_2(X_2)=z^2+w^2$ and $\alpha_2(Y_2)=zw$. 
We have $k[\V_2\oplus \V_2]^{SL_2(k)}=k[f_1,f_2,f_3]\cong k^{[3]}$ where:
\[
f_1=2X_0X_2-X_1^2 \, ,\,\, f_2=X_0Y_2-X_1Y_1+X_2Y_0\, ,\,\, f_3=2Y_0Y_2-Y_1^2
\]
See \cite{Freudenburg.17}, Section 8.7. We find that $\alpha_2(f_1)=4$, $\psi_2(f_2)=0$ and $\alpha_2(f_3)=-1$. Therefore $\krn\alpha_2= (f_1-4, f_2,f_3+1)$ and:
\[
B^{\Z_2}=k[z^2+w^2,xz+yw,x^2+y^2,zw, xw+yz,xy]\cong k[\V_2\oplus \V_2]/(f_1-4,f_2,f_3+1)
\]
This is a fiber of the quotient for $\V_2\oplus\V_2$ and gives the coordinate ring for $PSL_2(k)=SL_2(k)/\Z_2$. 
}
\end{example}

\begin{example}\label{n=3} {\rm $H=\Z_3$ and $B^H=k[\widehat{x^3+y^3}]$ has four generators. 
Coordinates for $\V_3$ are given by $(X_0,X_1,X_2,X_3)$ and the equivariant surjection $\beta_3:k[\V_3]\to B^H$ is defined by $\beta_3(X_3)=z^3+w^3$. 
We have $k[\V_3]^{SL_2(k)}=k[h]\cong k^{[1]}$ where $h$ is the discriminant: 
\[
h=8X_0X_2^3-3X_1^2X_2^2+9X_0^2X_3^2-18X_0X_1X_2X_3+6X_1^3X_3
\]
See \cite{Freudenburg.17}, Section 8.7. We find that $\beta_3(h)=324$ and $\krn\beta_3=(h-324)$. Since $h$ is homogeneous of degree 4, it follows that: 
\[
B^{\Z_3}=k[z^3+w^3, xz^2+yw^2, x^2z+y^2w, x^3+y^3]\cong k[\V_3]/(h-1)
\]
In particular, the homogeneous space $SL_2(k)/\Z_3$ is a hypersurface which is a fiber of the quotient for $\V_3$. }
\end{example}

\begin{example} {\rm $H={\rm BD}_8=\langle \vec{i},\vec{j},\vec{k}\rangle$ and $B^H=k[\widehat{x^4+y^4}]$ has five generators. 
Coordinates for $\V_4$ are given by $(X_0,X_1,X_2,X_3,X_4)$ and the equivariant surjection $\gamma_1:k[\V_4]\to B^H$ is defined by $\gamma_1(X_4)=z^4+w^4$. 
We have $k[\V_4]^{SL_2(k)}=k[q,r]\cong k^{[2]}$ where $q=T_4=2X_0X_4-2X_1X_3+X_2^2$ and:
\[
r=12X_0X_2X_4-6X_1^2X_4-9X_0X_3^2+6X_1X_2X_3-2X_2^3
\]
See \cite{Freudenburg.17}, Section 8.7. We find that $\gamma_1(q)=48$ and $\gamma_1(r)=0$. Since $q$ is homogeneous of degree 2, it follows that:
\[
B^{{\rm BD}_8}\cong k[\V_4]/(q-1,r)
\]
which is a fiber of the quotient for $\V_4$. }
\end{example}

 \begin{theorem}\label{strong} Each of the presentations {\rm (i)-(vii)} above is minimal.
\end{theorem}

\begin{proof} In each case, assume that $\varphi :k[W]\to B^H$ is a minimal presentation.
Write $W=\bigoplus_{i\in I}\V_{d_i}$. If $A_{d_i}=\{ 0\}$ for some $i\in I$, then 
$\varphi (\V_{d_i})=\{ 0\}$ by equivariance. But then $W/\V_{d_i}$ would give a presentation of smaller dimension, a contradiction. Therefore, 
$A_{d_i}\ne\{ 0\}$ for each $i\in I$.  

Case(i). $n\ge 3$ is odd and 
$A=k[x^n,xy,y^n]$ so $A_d=\{ 0\}$ for odd $d<n$ and $A_n\ne\{ 0\}$. Therefore, $d_i$ is odd for at least one $i\in I$ (otherwise $A_d=\{ 0\}$ for every odd $d$). If $d_i<n$ for odd $d_i$ then $A_{d_i}=\{ 0\}$, which gives a contradiction. 
So we must have $d_i=n$, which gives $W=\V_n$ and $\gamma_n$ is minimal. 

Case (ii). We first consider the case $n=2$. Suppose that $3\le \dim W\le 5$. Since $n$ is even, the  summands of $W$ are among $\V_0,\V_2,\V_4$. 
If $\V_4$ is a submodule of $W$ then $W=\V_4$, which gives a contradiction, 
since $A_2=\{ 0\}$ in this case. So $W=a\V_0\oplus\V_2$ for $0\le a\le 2$. 
But then $A_2$ is a one-dimensional vector space over $k$, which gives a contradiction, 
since $A_2=kx^2\oplus kxy\oplus ky^2$. 
So $\dim W=6$ and $\alpha_2$ is minimal. 

Consider the case $n=2m$ for even $m\ge 2$. Then $d_i$ is even for each $i\in I$. Suppose that $i\in I$ is such that $d_i=2m_i<n$. Let $(X_j)$ be coordinates for $\V_{di}$. Then $\varphi (X_0)\in A_{d_i}=k(xy)^{m_i}$, which implies $\varphi (X_j)\in k[\widehat{xy}]$ for $0\le j\le d_i$. 
Therefore:
\begin{equation}\label{less-than}
d_i<n \implies \varphi (\V_{d_i})\subset k[\widehat{xy}]
\end{equation}
Since elements of $\V_{d_i}$ for $i\in I$ generate $k[W]$ as a $k$-algebra, it follows that $d_i\ge n$ for at least one $i$. 

Suppose that $|I|\ge 3$. In this case, we have
\[
\sum_{i\in I}d_i+3\le \sum_{i\in I}d_i+|I|=\dim_kW\le\dim_k(\V_n\oplus\V_2)=n+4
\]
which implies $d_i<n$ for each $i\in I$, a contradiction. Therefore, 
$|I|\le 2$.

Suppose that $|I|=1$. Then $W=\V_N$ for some even integer $N$. 
From line (\ref{less-than}) we have $N\ge n$ and, since $\dim_kW\le \dim_k(\V_n\oplus\V_2)$, we also have $N\le n+2$. So either $N=n$ or $N=n+2$.

If $N=n$, then there exists 
$f\in A_n=kx^{2m}+k(xy)^m+ky^{2m}$ such that $B^{\Z_n}=k[\widehat{f}]$.
If $N=n+2$ then there exists 
$f\in A_{n+2}=(xy)^2A_n$ such that $B^{\Z_n}=k[\widehat{f}]$.
In either case, since $m$ is even, 
\[
f\in \left( B^{\Z_{2m}}\right)^{\vec{j}}=B^{\rm BD_{4m}} 
\quad\text{where}\quad 
\vec{j}=\begin{pmatrix} 0&1\cr -1 & 0\end{pmatrix}
\]
acts by right multiplication and $\vec{j}(xy)=-xy$. But then $xy\notin k[\widehat{f}]$ gives a contradiction. 

Therefore, $|I|=2$. In this case, $d_1+d_2+2\le n+4$ implies $d_1+d_2\le n+2$. Assuming $d_1\ge d_2$, we must have $d_1\ge n$. Therefore, $d_1=n$ and $d_2=2$, and we see that $W=\V_n\oplus \V_2$ and $\alpha_n$ is minimal. 

Case (iii). Arguing as in case (ii), we find that $d_i\ge n$ for at least one $i$. But then $W=\V_n$ and $\beta_n$ is minimal. 

Case (iv). 
Arguing as in case (ii), we find that $d_i\ge 2n$ for at least one $i$. But then 
$W=\V_{2n}$ and $\delta_n$ is minimal. 

Case (v). We have $A=\C [v_1,v_2,v_3]$ in degrees 8, 12 and 6. 
If $d_i<6$ for some $i\in I$, then $A_{d_i}=\{ 0\}$, which gives a contradiction. 
Therefore, $W=\V_6$ and $\epsilon$ is minimal. 

Case (vi). We have $A=\C [v_1,v_2v_3,v_3^2]$ in degrees 8, 18 and 12. 
If $d_i<8$ for some $i\in I$, then $A_{d_i}=\{ 0\}$, which gives a contradiction. 
Therefore, $W=\V_8$ and $\eta$ is minimal. 

Case (vii). We have $A=\C [w_1,w_2,w_3]$ in degrees 12, 20 and 30. 
If $d_i<12$ for some $i\in I$, then $A_{d_i}=\{ 0\}$, which gives a contradiction. 
Therefore, $W=\V_{12}$ and $\kappa$ is minimal. 
\end{proof}


\section{Absolutely nonextendable $SL_2$-action on $X$}\label{SL2-nonextend}

The purpose of this section is to give a proof of {\it Theorem\,\ref{non-extend1}}.  Let $k=\C$ and let 
$\mathcal{Q}\subset k[\V_2]=k[X_0,X_1,X_2]$ be the vector space of ternary quadratic forms. Then $\dim\mathcal{Q}=6$ and by
{\it Corollary\,\ref{quadratics}} we have:
\[
k[\mathcal{Q}]=k[\widehat{T}_0,\widehat{T}_2] \quad\text{where}\,\, T_0=X_0^2 \,\, ,\,\, T_2=2X_0X_2-X_2^2
\]
The fundamental pair $(D_2,U_2)$ restricts to $k[\mathcal{Q}]$. 
Let $R'=k[\V_2]/(T_2-1)=k[x_0,x_1,x_2]$. Define $\pi :k[\V_2]\to R'$, $\pi (X_i)=x_i$, and set 
$R=\pi (k[\mathcal{Q}])$. We have $Y\cong{\rm Spec}(R')$ and $X\cong {\rm Spec}(R)$. 
Since $\deg x_0^2=4$ we obtain the equivariant surjection 
\begin{equation}\label{phi-presentation}
\varphi :k[\V_4]\to R=k[\widehat{x_0^2}]
\end{equation}
defined by $\varphi (Y_4)=x_2^2$, where $k[\V_4]=k[Y_0,Y_1,Y_2,Y_3,Y_4]$. 

Let $(\delta ,\upsilon )$ be the induced fundamental pair on $R$, let $A=\krn\delta$, and let
$A=\bigoplus_{d\in\N}A_d$ be the $\N$-grading defined by $(\delta ,\upsilon)$.
By equivariance we have:
\[
\varphi (Y_0)\in A_4=kx_0^2\,\, ,\,\, \varphi (T_2(Y_i))\in A_4=kx_0^2\,\, ,\,\, \varphi (T_4(Y_i))\in A_0=k
 \]
 We find by direct computation that $\varphi (Y_0)=8x_0^2$, $\varphi (T_2(Y_i))=16 x_0^2$ and $\varphi (T_4(Y_i))=4$.
 Therefore, $2Y_0-T_2(Y_i), T_4(Y_i)-4\in\krn\varphi$. It is easy to check that
 \begin{equation}\label{kernel}
 \krn\varphi = (\widehat{2Y_0-T_2(Y_i)}, T_4(Y_i))
 \end{equation}
 where $T_2=2Y_0Y_2-Y_1^2$ and $T_4=2Y_0Y_4-2Y_1Y_3+Y_0^2$. 
It follows that $\krn\varphi$ contains
 \begin{eqnarray*}
\textstyle \frac{1}{2}U_4^2(2Y_0-T_2(Y_i)) -T_4(Y_i)&=& 2Y_2-(2Y_0Y_4+Y_1Y_3-Y_2^2)-(2Y_0Y_4-2Y_1Y_3+Y_2^2)\\
 &=& 2Y_2-4Y_0Y_4+Y_1Y_3+4
 \end{eqnarray*}
 which is a variable of $k[\V_4]\cong k^{[5]}$. Therefore, $X$ admits a closed algebraic embedding in $\C^4$. 
 
 Since $\krn\delta \cap\krn\upsilon = k$ we see that $ML(R)=k$.  
 If $X$ admits a closed algebraic embedding in $\C^3$, then {\it Theorem\,\ref{Band-ML}} implies that $X$ is a hypersurface of the form
 $xy=P(z)$ for some polynomial $P$. However, it is shown in \cite{Freudenburg.24}, Proposition 5.3, that $X$ is not isomorphic to such a surface. 
 Therefore, the embedding dimension of $X$ equals 4. This proves part (b) of {\it Theorem\,\ref{non-extend1}}. 
 
 Part (c) is implied by the following result.
 
\begin{lemma}\label{V4} 
Let $B$ be an affine $k$-domain, where $k=\bar{k}$ and $\dim B\ge 1$. Assume that each of $B$ and $k^{[4]}$ is endowed with  
an algebraic action of $SL_2(k)$. The following conditions are mutually exclusive.
\begin{enumerate}
\item There exists an equivariant surjection $\varphi :k[\V_4]\to B$. 
\item There exists an equivariant surjection $\pi :k^{[4]}\to B$. 
\end{enumerate}
\end{lemma}

\begin{proof} 
Let $(D_4,U_4)$ be the basic fundamental pair on $k[\V_4]$, and let $(D,U)$ be the given fundamental pair on $B$. 
Let $\mathcal{A}=\krn D_4$ and $A=\krn D$
with induced $\N$-gradings $\mathcal{A}=\bigoplus_{d\in\N}\mathcal{A}_d$ and $A=\bigoplus_{d\in\N}A_d$.
Then $\mathcal{A}_d=0$ for odd $d$ and for $d=2$; see \cite{Freudenburg.17}, Example 8.12. 

Assume that condition (1) holds.
Since $A_d=\varphi (\mathcal{A}_d)$ for all $d\in\N$, it follows that:
\[
(\dag ) \hspace{.1in} A_d=0 \,\,\text{\it for odd $d$ and for}\,\, d=2
\]

Assume that condition (2) holds. Let $(\tilde{D},\tilde{U})$ be the given fundamental pair for $k^{[4]}$ with $\tilde{A}=\krn\tilde{D}$. 
According to Panyushev \cite{Panyushev.84}, every algebraic $SL_2$-action on $\A_k^4$ is equal (in some coordinate system) 
to the action induced by an $SL_2$-module $W$. There are five $SL_2$-modules of dimension four, namely:
\[
\V_3\,\, ,\,\, 2\V_1\,\, ,\,\, \V_0\oplus \V_2\,\, ,\,\, 2\V_0\oplus \V_1\,\, ,\,\, 4\V_0
\]
We consider each of these cases, using the fact that $A=\pi (\tilde{A})$. 
\begin{enumerate}[label=(\roman*)]
\item If $W=\V_3$, then $\tilde{A}$ is generated in degrees 2 and 3. Therefore, either $A_2\ne 0$ or $A_3\ne 0$.
\item If $W=2\V_1$ or $W=2\V_0\oplus\V_1$, then $\tilde{A}$ is generated in degree 1. Therefore, $A_1\ne 0$.
\item If $W=\V_0\oplus\V_2$, then $\tilde{A}$ is generated in degree 2. Therefore, $A_2\ne 0$.
\item If $W=4\V_0$ then $A=B$.
\end{enumerate}
We see that $\{ {\rm (i)\vee (ii)\vee (iii)\vee (iv)}\}\wedge \{ (\dag )\}=\emptyset$. 
\end{proof}

\begin{corollary} The presentation $\varphi :k[\V_4]\to \mathcal{O}(X)$ defined in {\rm (\ref{phi-presentation})} is minimal.
\end{corollary}

Part (a) is implied by this corollary. This completes the proof of {\it Theorem\,\ref{non-extend1}}. 


\section{The cylinder over $X$}\label{cylinder}

In this section, assume that $k=\bar{k}$. 
Let $(\delta^{\prime},\upsilon^{\prime})$ be the nontrivial fundamental pair on $\mathcal{O}(Y)$.
This restricts to the fundamental pair $(\delta ,\upsilon )$ on $\mathcal{O}(X)$. 
Let $T\subset SL_2(k)$ be the torus, $T\cong k^*$. 
The {\bf triangular subgroup} $\mathcal{T}$ of ${\rm Aut}_k(\mathcal{O}(X))$ is generated by $T$ and automorphisms $\alpha$ preserving $\deg_{\delta}$,  i.e.,
$\alpha (\krn (\delta^n))=\krn (\delta^n)$ for each $n\in\N$. Likewise, 
the {\bf triangular subgroup} $\mathcal{T}^{\prime}$ of ${\rm Aut}_k(\mathcal{O}(Y))$ is generated by $T^{\prime}$ and automorphisms preserving $\deg_{\delta^{\prime}}$.
The description of the groups ${\rm Aut}_k(Y)$ and ${\rm Aut}_k(X)$ given by Danilov and Gizatullin \cite{Danilov.Gizatullin.77}
can be recast in the following form.  

\begin{theorem} With the assumptions above:
\begin{itemize}
\item [{\bf (a)}] ${\rm Aut}_k(\mathcal{O}(Y))\cong O_3(k)\ast_{H^{\prime}}\mathcal{T}^{\prime}$ where $H^{\prime}=O_3(k)\cap\mathcal{T}^{\prime}$
\item [{\bf (b)}] ${\rm Aut}_k(\mathcal{O}(X))\cong PSL_2(k)\ast_H\mathcal{T}$ where $H=PSL_2(k)\cap\mathcal{T}$
\end{itemize}
\end{theorem}
The amalgamated free product structure of these groups implies that every $\G_a$-action on $Y$ can be conjugated to $\mathcal{T}^{\prime}$, and 
every $\G_a$-action on $X$ can be conjugated to $\mathcal{T}$; see \cite{Serre.80}. 
Moreover, by definition, every nontrivial triangular $\G_a$-action has the same ring of invariants, since the invariant ring is uniquely determined by the given degree function. 
We conclude that
\begin{enumerate}
\item for every nonzero $D^{\prime}\in {\rm LND}(\mathcal{O}(Y))$, there exists $\alpha^{\prime}\in{\rm Aut}_k(\mathcal{O}(Y))$ such that \\
$\krn D^{\prime}=\alpha^{\prime} (\krn\delta^{\prime})$, and 
\item for every nonzero $D\in {\rm LND}(\mathcal{O}(X))$, there exists $\alpha\in{\rm Aut}_k(\mathcal{O}(X))$ such that \\ $\krn D=\alpha (\krn\delta )$. 
\end{enumerate}
This observation yields the following. 
\begin{corollary}\label{plinth} With the assumptions above:
\begin{itemize}
\item [{\bf (a)}] Every irreducible $D^{\prime}\in {\rm LND}(\mathcal{O}(Y))$ is conjugate to $\delta^{\prime}$ and $\krn (D^{\prime})/{\rm pl}(D^{\prime})\cong k$. 
\item [{\bf (b)}] Every irreducible $D\in {\rm LND}(\mathcal{O}(X))$ is conjugate to $\delta$ and $\krn (D)/{\rm pl}(D)\cong k$. 
\end{itemize}
\end{corollary}

\begin{proof} Let $D^{\prime}\in {\rm LND}(\mathcal{O}(Y))$ be irreducible and let $\alpha\in{\rm Aut}_k(\mathcal{O}(Y))$ be such that:
\[
\krn\delta^{\prime}=\alpha (\krn D^{\prime})=\krn (\alpha D^{\prime}\alpha^{-1})
\]
Set $E=\alpha D^{\prime}\alpha^{-1}$. Then $E$ is irreducible and $A:=\krn E=\krn\delta^{\prime}\cong k^{[1]}$. 
By \cite{Freudenburg.17}, Principle 12, there exist nonzero $a,b\in A$ with $a\delta^{\prime}=bE$. 
We may assume that $\gcd_A(a,b)=1$.
Let $f,g\in A$ be such that $af+bg=1$. We have:
\[
ga\delta^{\prime}=gbE=(1-af)E \implies E=a(g\delta^{\prime}-fE)
\]
Since $E$ is irreducible as a derivation it follows that $a\in k^*$. Irreducibility of $\delta^{\prime}$ now implies $b\in k^*$. 
So $E=c\delta^{\prime}$ for some $c\in k^*$. 

Let $T\to {\rm Aut}_k(\mathcal{O}(Y))$, $t\to \lambda_t$, be the torus action which is the restriction of the $SL_2$-action. 
Since $\deg\delta^{\prime}$ is homogeneous of degree 2, we see that $\lambda_t\delta^{\prime}\lambda_t^{-1}=t^2\delta^{\prime}$. 
Choose $t\in k^*$ so that $t^2=c$. Then $E=c\delta^{\prime}=\lambda_t\delta^{\prime}\lambda_t^{-1}$. 
So $E$ and $\delta^{\prime}$ are conjugate, which implies that $D^{\prime}$ and $\delta^{\prime}$ are conjugate. 
Moreover, in $\mathcal{O}(Y)=k[x_0,x_1,x_2]$ we have $\krn\delta^{\prime}=k[x_0]$ and $\delta^{\prime}x_1=x_0$, which implies that:
\[
{\rm pl}(\delta^{\prime})=(x_0)\quad\text{and}\quad \krn (\delta^{\prime})/{\rm pl}(\delta^{\prime})\cong k
\]
This proves part (a), and the proof of part (b) follows exactly the same line of reasoning. 
\end{proof}

Let $R=\mathcal{O}(X)$. 
The preceding section shows that $R$ is generated by quadratic forms in 
$\mathcal{O}(Y)=k [x_0,x_1,x_2]$. 

 \begin{theorem}\label{non-cancel}
 Let $R[T]\cong_kR^{[1]}$ and define 
 \[
 F=15(x_0^2)T-3(x_0x_1)(x_2^2)-2(x_1x_2) \quad \text{and}\quad 
 \tilde{R}=R[T]/(F)
 \]
 \begin{itemize}
 \item [{\bf (a)}] $R^{[1]}\cong_k\tilde{R}^{[1]}$
 \smallskip
 \item [{\bf (b)}] $R\not\cong_k\tilde{R}$
 \end{itemize}
 \end{theorem}

\begin{proof} Since 
\[
1=2x_0x_2-x_1^2\implies x_0x_2=2(x_0^2)(x_2^2)-(x_0x_1)(x_1x_2)
\]
it follows that $R=k[x,y,z,w]$ where:
\[
x=x_0^2\,\, ,\,\, y=x_0x_1\,\, ,\,\, z=x_1x_2\,\, ,\,\,w=x_2^2
\]
Define $P=x_0x_2=2xw-yz$. From equation (\ref{kernel}) above we find that the ideal $I$ of relations in $R$ is
\[
I=\left( y^2-x(2P-1) \, ,\,  yP-xz \, , \, yz-P(2P-1) \, , \, yw-zP \, , \, z^2-w(2P-1)\right)
\]
The fundamental pair on $R$ is $(\delta ,\upsilon)$ and $\delta$ operates on $R$ by:
\[
\delta w=2z\,\, ,\,\, \delta z=3P-1\,\, ,\,\, \delta P=y\,\, ,\,\, \delta y=x\,\, ,\,\, \delta x=0
\]
Note that $\delta (zP)=5P^2-2P$. 

Extend $\delta$ to $\Delta$ on $R[v]=R^{[1]}$ by $\Delta v=w$. 
Define $Q=5xv-yw$. Then $\Delta Q=2P$, 
so $\Delta s=1$ for $s=\frac{3}{2}Q -z$. Set $\tilde{R}=\krn\Delta$. By the Slice Theorem,
\[
R^{[1]}=R[v]=k[x,y,z,w,v]=k[x,y,s,w,v]=\tilde{R}[s]=\tilde{R}^{[1]}
\]
and
\[
\tilde{R}\cong_kR[v]/(s)=k[\bar{x},\bar{y},\bar{w},\bar{v}]
\]
where $2\bar{z}=3\bar{Q}$ and $2\bar{P}=4\bar{x}\bar{w}-3\bar{y}\bar{Q}$. 
This proves part (a).

Let $L=k[\bar{x},\bar{x}^{-1}]$ and observe that $\tilde{R}[\bar{x}^{-1}]=L[\bar{y},\bar{w},\bar{v}]=L[\bar{y},\bar{P},\bar{Q}]$. 
Since $3\bar{x}\bar{Q}=2\bar{y}\bar{P}$ and $\bar{x}(2\bar{P}-1)=\bar{y}^2$ we see that $\tilde{R}[\bar{x}^{-1}]=L[\bar{y}]=L^{[1]}$. Define $\Theta\in{\rm LND}(L[\bar{y}])$ 
by $\Theta \bar{y}=\bar{x}$. From the observed relations we find:
\[ \textstyle
\Theta \bar{P}=\bar{y}\,\, ,\,\, \Theta \bar{Q}=2\bar{P}-\frac{2}{3} \,\, ,\,\, \Theta \bar{w}= 3\bar{Q} \,\, ,\,\, \Theta \bar{v}=\bar{w}-\frac{2}{15}\bar{x}^{-1}
\]
Define $\theta=\bar{x}\,\Theta\in{\rm LND}(\tilde{R})$. Let $J=(\theta\tilde{R})$, the ideal generated by the image of $\theta$. Then
\[
J=(\theta \bar{x},\theta \bar{y},\theta \bar{w},\theta \bar{v})=(\bar{x}^2,\bar{x}\bar{Q},\bar{x}\bar{w}-\textstyle\frac{2}{15})
\]
Modulo $J$, we have:
\[
0\equiv \bar{x}\bar{Q}\equiv 5\bar{x}^2\bar{v}-\bar{x}\bar{y}\bar{w}\equiv\textstyle\frac{2}{15}\bar{y} \implies
0\equiv 2\bar{P}-\textstyle\frac{2}{3}\equiv 4\bar{x}\bar{w}-2\bar{y}\bar{z}-\frac{2}{3}\equiv -\frac{2}{15} \implies J=(1)
\]
Therefore, $\theta$ is irreducible and the $\G_a$-action on ${\rm Spec}(\tilde{R})$ induced by $\theta$ is fixed-point free. 
We thus obtain:
\begin{enumerate}
\item $\theta\in{\rm LND}(\tilde{R})$ is irreducible.
\item $\krn\theta =k[\bar{x}]$
\item $\theta \bar{y}=\bar{x}^2$ and ${\rm pl}(\theta )=\bar{x}^2k[\bar{x}]$
\end{enumerate}
By {\it Corollary\,\ref{plinth}}, there is no irreducible $D\in{\rm LND}(R)$ such that $\krn D/{\rm pl}(D)\cong k[x]/(x^2)$. Therefore,
$\tilde{R}\not\cong_kR$.
\end{proof}


\section{Closed embeddings of $Y$ in $\A_k^3$}\label{embedding}

The purpose of this section is to prove {\it Theorem\,\ref{two-embed}}. 
Let $B=k[x,y,z]=k^{[3]}$. Given $f,g\in B$
define the $k$-derivation $\Delta_{(f,g)}$ of $B$ by
\[ 
\Delta_{(f,g)}(h)=\frac{\partial (f,g,h)}{\partial(x,y,z)} \,\, ,\,\, h\in B 
\]
that is, the determinant of the Jacobian matrix of $f,g$ and $h$. We have:
\[
F=xz-y^2 \,\, ,\,\, r=yF-x\,\, ,\,\, G=x^{-1}(F^4+r^3)
\]
It is easy to check that $G+r^2\in FB$. Set $H=F^{-1}(G+r^2)$. 
It is also easy to check that $G$ and $H$ are irreducible. 
Define $k$-derivations $\Delta_i$ of $B$, $1\le i\le 3$ by:
\[
\Delta_1=\Delta_{(x,F)}\,\, ,\,\, \Delta_2=\Delta_{(F,G)}\,\, ,\,\, \Delta_3=\Delta_{(G,H)}
\]
The following properties hold.
\begin{enumerate}
\item $\Delta_1$ is locally nilpotent, $\krn \Delta_1=k[x,F]$ and $\Delta_1r=xF$.
\item $\Delta_2$ is locally nilpotent, $\krn \Delta_2=k[F,G]$ and $\Delta_2r=FG$.
\item $\Delta_3$ is locally nilpotent, $\krn \Delta_3=k[G,H]$ and $\Delta_3r=GH$.
\end{enumerate} 
Statement (1) is well known. Since $xG=F^4+r^3$, 
\cite{Freudenburg.17}, Theorem 5.28, implies that $\Delta_2$ is locally nilpotent, $\krn \Delta_2=k[F,G]$ and $\Delta_2r=FG$. So statement (2) holds.
Likewise, since $FH=G+r^2$,
the just cited theorem implies that $\Delta_3$ is locally nilpotent, $\krn \Delta_3=k[G,H]$ and $\Delta_3r=GH$. So statement (3) holds.

Given $a\in k^*$ and $b\in B$, let $\bar{B}=B/(G-a)$ and let $\bar{b}$ denote the image of $b$ in $\bar{B}$. 
It is easy to check that ${\rm Spec}(\bar{B})$ is smooth, so $\bar{B}$ is a normal ring. 
In addition, $\Delta_2,\Delta_3$ induce $\delta_2,\delta_3\in {\rm LND}(\bar{B})$, where $\delta_2\bar{r}=a\bar{F}$ and $\delta_3\bar{r}=a\bar{H}$, and
$\bar{F}\bar{H}=\bar{r}^2+a$. Suppose that $\dim ML(\bar{B})=1$. 
Then $\krn\delta_2=\krn\delta_3$ implies $\delta_2\bar{H}=0$. But then 
\[
0=\delta_2(a)=\delta_2(\bar{F}\bar{H}-\bar{r}^2)=-2\bar{r}\delta_2(\bar{r})=-2a\bar{r}\bar{F}
\]
gives a contradiction. Therefore, $ML(\bar{B})=k$. 
By \cite{Freudenburg.17}, Lemma 2.34, $\krn\theta\cong k^{[1]}$ for every nonzero $\theta\in {\rm LND}(\bar{B})$. 
By Daigle's characterization of Danielewski surfaces (\cite{Daigle.04}, Theorem 2.5), $\bar{B}$ is a Danielewski surface and $\bar{B}=k[\bar{F},\bar{H},\bar{r}]$. Since  
$\bar{F}\bar{H}=\bar{r}^2+a$ is a prime relation, 
it follows that
$\bar{B}$ is isomorphic as a $k$-algebra to $B/(xz-y^2-a)$.

Regarding zero fibers, the surface defined by $G=0$ in $\A_k^3$ contains a line of singular points, and is not a normal surface. 
In addition, it is well-known that every fiber of the polynomial $xz-y^2$ is a normal surface. 

This completes the proof of {\it Theorem\,\ref{two-embed}}.


\section{Questions and remarks}

Let $X=SL_2(\C)/N$, as above. 

\begin{question} Let $H\subset SL_2(\C )$ be isomorphic to either $\G_a$ or a finite nonabelian group. Is the induced $H$-action on the surface $X$ the restriction of an $H$-action on $\C^4$, linear or otherwise?
\end{question}

Note that the triangular variable $2Y_2-4Y_0Y_4+Y_1Y_3+4$ used to define the embedding of $X$ in $\C^4$ is an invariant of the linear $\C^*$-action on $\C^5$ coming from the minimal presentation of $X$. 
This means that the $\C^*$-action on $X$ induced by the $SL_2$-action is the restriction of a $\C^*$-action on $\C^4$.  

Likewise, the $\C^*$-action used in the example for $\C^*\times\C^1$ is the restriction of a linear $\C^*$-action on $\C^3$. 
\begin{question}
Is the $\Z_2$-action on $\C^*\times\C^1$ defined by $\mu =(t^{-1},t^ns)$ for odd $n\ge 1$ absolutely nonextendable.? 
\end{question}

\begin{question} 
Can the $SL_2$-action on $X$ be extended holomorphically to $\C^4$? 
\end{question}

\begin{question} Are minimal presentations unique, up to a $G$-module isomorphism?
\end{question}


\vspace{.2in}

\noindent \address{Department of Mathematics\\
Western Michigan University\\
1903 W. Michigan Ave.\\
Kalamazoo, Michigan 49008} \,\,USA\\
\email{gene.freudenburg@wmich.edu}

\end{document}